\documentclass[10pt, a4paper]{amsart}

\usepackage{amsmath,amssymb}

\numberwithin{equation}{section}

\usepackage{hyperref}

\usepackage[abbrev]{amsrefs}
\usepackage{enumerate}
\usepackage[T1]{fontenc}
\usepackage{microtype}

\usepackage[english]{babel}

\newtheorem{proposition}{Proposition}[section]

\theoremstyle{definition}

\newtheorem{remark}{Remark}[section]

\newcommand{\norm}[1]{\mathopen\Vert#1\mathclose\Vert}

\DeclareMathOperator{\dist}{dist}

\newcommand{\N}{{\mathbb N}}

\newcommand{\R}{{\mathbb R}}

\newcommand{\eps}{\varepsilon}
\renewcommand{\epsilon}{\eps}

\title[Critical points whose polarization is a critical point]{Finding critical points whose \\ polarization is also a critical point}
\author{Marco Squassina}
\address{Universit\`a degli Studi di Verona \\ Dipartimento di Informatica \newline C\`a Vignal 2\\ Strada Le Grazie 15\\
I-37134 Verona\\ Italy}
\email{marco.squassina@univr.it}

\author{Jean Van Schaftingen}
\address{Universit\'e catholique de Louvain\\
Institut de Recherche en Math\'ematique et Physique (IRMP)\newline
Chemin du Cyclotron 2 bte L7.01.01\\
1348 Louvain-la-Neuve\\
Belgium}
\email{Jean.VanSchaftingen@uclouvain.be}


\begin{document}

\begin{abstract}
We show that near any given minimizing sequence of paths for the mountain pass lemma, 
there exists a critical point whose polarization is also a critical point. This is 
motivated by the fact that if any polarization of a critical point is also a 
critical point and the Euler-Lagrange equation is a second-order semi-linear 
elliptic problem, T.\thinspace Bartsch, T.\thinspace Weth and M.\thinspace Willem 
(J. Anal. Math., 2005) have proved that the critical point is axially symmetric.
\end{abstract}

\maketitle


\section{Introduction}
\noindent
If \(u : \Omega \to \R\) solves the semi-linear elliptic problem
\begin{equation}
\label{problem}
\begin{cases}
 - \Delta u   = f (x, u)  &  \text{in \(\Omega\)},\\
 \,\,\, u  = 0 &  \text{on \(\partial \Omega\)},
\end{cases}
\end{equation}
one is interested in determining whether \(u\) inherits some symmetry of 
the domain \(\Omega \subset \R^N\) and of the nonlinearity \(f\). For example 
if \(\Omega\) and \(f\) are invariant under rotations, is \(u\) also invariant?
Of course, when \(u\) is the only solution of \eqref{problem}, the answer is positive.
By observing the eigenfunctions of the Laplacian, one can see that this is not always the case. 
B.\thinspace Gidas, W.-M.\thinspace Ni and L.\thinspace Nirenberg have 
proved that if \(\Omega\) is a ball, \(f\) is independent of \(x\) 
and Lipschitz-continuous and \(u\) is positive, then \(u\) is radially symmetric \cite{GidasNiNirenberg1979}. 
The main tool in the proof is the maximum principle for second order elliptic operators.
One can try to replace the essential positivity assumption by some other assumption. O.\thinspace Lopes 
has proved that if the solution \(u\) is a minimizer under a constraint, if \(\Omega\) is bounded and smooth and \(f\) is smooth 
enough, then \(u\) is radially symmetric \cite{Lopes1996}. His proof relies on a unique continuation principle.
Another family of methods is based on the symmetrization by rearrangement. The first 
idea is to associate to any nonnegative measurable function \(u : \Omega \to \R\) its 
Schwarz symmetrization \(u^*\) which is a radial function such that the corresponding sub-level 
sets have the same measure as those of \(u\); under this transformations, the \(L^2\)-norm 
of the gradient decreases \citelist{\cite{PolyaSzego}\cite{Kawohl1985}}. In particular, it 
is possible to show that many functionals of the calculus of variations decrease under symmetrization, and therefore that 
if \(u\) is a solution of some variational problem, then \(u^*\) is also a solution. However 
this does not imply that \(u\) itself is symmetric. One way to show that \(u\) is symmetric is 
to study the equality cases of symmetrization inequalities \cite{BrothersZiemer1988}; this 
approach is however limited by some stringent assumptions to apply the results.
In order to study partial symmetry, T.\thinspace Bartsch, T.\thinspace Weth and M.\thinspace Willem, 
have introduced a nice method which mixes a variational argument with the maximum 
principle \cite{BartschWethWillem} (see also \cite{VanSchaftingenWillem}). Given a 
closed half-space \(H \subset \R^N\), define \(\sigma_H\) to be the reflection with 
respect to \(\partial H\). If \(\sigma_H(\Omega) = \Omega\), define for \(u : \Omega \to \R\) its polarization \(u^H : \Omega \to \R\) by
\[
 u^H =
\begin{cases}
 \max \{u, u \circ \sigma_H\} & \text{on \(H\)},\\
 \min \{u, u \circ \sigma_H\} & \text{on \(\R^N \setminus H\)}.
\end{cases}
\]
Now assume that \(u\) is a minimizer of some functional. Then, if the functional does not 
increase under polarization with respect to \(H\), it follows that \(u^H\) is a minimizer too. Since symmetrization 
can be approximated by rearrangement \cite{BrockSolynin2000} (see also \cite{VanSchaftingen2009}), this is stronger than requiring that the 
functional does not increase under symmetrization. The new ingredient that T.\thinspace Bartsch, T.\thinspace 
Weth and M.\thinspace Willem introduce is that if \(\Omega\) is a ball and \(u^H\) is also a solution for every 
half-space \(H\) such that \(\sigma_H(\Omega) = \Omega\), then \(u\) is axially symmetric.
The method applies to minimizers under constraints and in particular to least energy solutions and least 
energy nodal solutions of semi-linear equations \cite{BartschWethWillem}*{Theorem 3.2}.
It would be nice to extend such results to critical points that are not minimizers under a constraint. 
One way to construct such critical points is to rely on the Mountain Pass lemma of A. Ambrosetti and P. Rabinowitz \cite{AmbrosettiRabinowitz}.
Given a functional \(\varphi \in C^1(H^1_0(\Omega))\) such that \(0\) is a local minimum of \(\varphi\), set
\[
  \Gamma=\{ \gamma \in C([0, 1], H^1_0(\Omega)):\, \gamma(0) = 0 \text{ and } \varphi(\gamma(1)) < 0\},
\]
and
\[
  c=\inf_{\gamma \in \Gamma} \sup_{t\in [0,1]} \varphi (\gamma(t)).
\]
Assume also that \(\varphi\) satisfies the Palais-Smale condition, that is, 
if \((u_n)_{n \in \N}\) is a sequence in \(H^1_0(\Omega)\) such 
that \((\varphi(u_n))_{n \in \N}\) converges and \(\varphi'(u_n) \to 0\) 
as \(n \to \infty\) in $H^{-1}(\Omega)$, then \((u_n)_{n \in \N}\) converges, up to a subsequence.
Then there exists \(u \in H^1_0(\Omega)\) such that \(\varphi' (u) = 0\) and \(\varphi (u) = c\).
If, in addition, \(\Omega\) is a ball and for {\em every} closed half-space \(H \subset \R^N\) such 
that \(\sigma_H(\Omega) = \Omega\) and \(u \in H^1_0(\Omega)\), \(\varphi (u^H) \le \varphi (u)\), then 
there exists \(u \in H^1_0(\Omega)\) such that \(\varphi' (u) = 0\), \(\varphi (u) = c\) and \(u\) is axially 
symmetric \cite{VanSchaftingen2005}. In general, it is not difficult to prescribe symmetry to solutions. The remarkable 
feature of this result is that \(u\) is a critical point at a critical level without any symmetry constraint. This result 
was extended to critical levels defined with the Krasnoselskiii genus \cite{VanSchaftingen2006} and to non-smooth critical 
point theory \cite{Squassina2011-1} (see also \citelist{\cite{Squassina2011}\cite{Squassina2012}}).
We would like to know when all the solutions obtained by the Mountain Pass lemma are symmetric. 
To this regard, we recall that the Mountain Pass value $c$ often coincides with the least energy value and for instance, in \cite{BJMar}, for a quite general 
class of autonomous functionals, the authors have recently proved that {\em any} least energy solution is radially symmetric and with fixed sign.
We also point out that symmetry results under assumptions on the Morse index and somewhat restrictive 
assumptions on the nonlinearity have been obtained in \citelist{\cite{PacellaWeth}\cite{GladialiPacellaWeth2010}}.
Going back to the minimax principle, we would like to apply the method of T.\thinspace Bartsch, T.\thinspace Weth 
and M.\thinspace Willem. The crucial step is to prove that if \(u\) is a critical point of $\varphi$ 
then \(u^H\) is also a critical point of $\varphi$. 
We could not prove this and we also think that this should not be true in general. 
However, we have something that we think to be the best result in that direction.

To state our result, recall that the critical points $u$ of the Mountain Pass 
lemma can be localized as follows: if \((\gamma_n)_{n \in \N}\) is a sequence 
of paths in \(\Gamma\) such that 
\begin{equation}
	\label{ip-cond}
\lim_{n  \to \infty} \sup_{t \in [0, 1]} \varphi(\gamma_n(t)) = c, 
\end{equation}
then, up to a subsequence, 
\begin{equation}
	\label{impl-adh}
\lim_{n \to \infty} \dist_{H^1}(u, \gamma_n([0, 1])) = 0.
\end{equation} 
If $\varphi$ does not increase under polarizations with respect to a fixed half-space $H$, 
based upon a new abstract minimax principle (Proposition~\ref{propositionAbstractSymmetricMinimax}), we 
prove that for any sequence \((\gamma_n)_{n \in \N}\) satisfying \eqref{ip-cond}, there exists a critical point \(u\) of $\varphi$ with $\varphi(u)=c$ such 
that, up to a subsequence, \eqref{impl-adh} holds and, {\em in addition}, \(u^H\) is also a critical point of $\varphi$ at the same level $c$ (Proposition~\ref{mpl}). 
This provides some kind of symmetry information of $u$ with respect to $H$, see e.g.\ \cite{BartschWethWillem}*{Theorem 2.6}
for the special situation regarding problem \eqref{problem}.
One can expect that in many cases, there is at most one critical point \(u\) such that 
$\dist_{H^1}(u, \gamma_n([0, 1])) \to 0$ as $n\to\infty$ and \(\varphi(u) = c\). In such a case we would have the desired property. 
Unfortunately, in general, the uniqueness of critical points at the 
level $c$ and near a family of paths seems quite difficult to establish. The result obtained also extends to continuous functionals
in the framework of the non-smooth critical point theory of \cites{DM,CDM}, by exploiting a suitable quantitative deformation theorem \cite{corvellec}.
\vskip2pt
The paper is organized as follows. In Section~\ref{sectionAbstract}, we 
prove a new quantitative abstract Minimax Principle. In Section~\ref{sectionMountain}, we apply 
this result in the specific case of the Mountain Pass lemma and the polarization.

\section{Shadowing Minimax Principle}

\label{sectionAbstract}

In this section we shall prove the following variant of the minimax 
principle in which two almost critical points related by a function \(\Psi\) are found at once.

\begin{proposition}
\label{propositionAbstractSymmetricMinimax}
Let $(X,\|\cdot\|)$ be a Banach space, \(M\) be a metric space and \(M_0 \subset M\). 
Let also consider \(\Gamma_0 \subset C(M_0, X)\) and define the set
\[
  \Gamma=\{ \gamma \in C(M, X): \gamma\vert_{M_0} \in \Gamma_0 \}
\]
If \(\varphi \in C^1(X, \R)\) 
satisfies
\[
  c=\inf_{\gamma \in \Gamma} \sup_{t\in M} \varphi (\gamma(t)) > \sup_{\gamma_0 \in \Gamma_0} \sup_{t \in M_0} \varphi(\gamma_0(t))=a,
\]
\(\Psi \in C(X, X)\) and 
\[
  \varphi \circ \Psi \le \varphi, \,\qquad \Psi(\Gamma)\subset\Gamma,
\]
then for every \(\epsilon \in ]0, \frac{c-a}{2}[\), \(\delta > 0\) and \(\gamma \in \Gamma\) such that 
\[
 \sup_{M} \varphi \circ \gamma \le c+\epsilon,
\]
there exist elements \(u, v, w \in X\) such that 
\begin{enumerate}[a)]
 \item[a.1)] \(c-2\epsilon \le \varphi(u) \le c+2\epsilon\),
\vskip2pt
 \item[a.2)] \(c-2\epsilon \le \varphi(v) \le c+2\epsilon\),
\vskip2pt
 \item[b.1)] \(\norm{u-w} \le 3\delta\),
\vskip2pt
 \item[b.2)] \(\dist_X(w, \gamma(M)) \le \delta\),
\vskip2pt
 \item[b.3)] \(\norm{v-\Psi(w)} \le 2\delta\), 
\vskip2pt
 \item[c.1)] \(\norm{\varphi'(u)}_{X'} < 8\epsilon/\delta\),
\vskip2pt
 \item[c.1)] \(\norm{\varphi'(v)}_{X'} < 8\epsilon/\delta\).
\end{enumerate}
\end{proposition}


\noindent
The proof relies on the following quantitative deformation lemma of M.\thinspace 
Willem \cite{WillemMinimax}*{Lemma 2.3}.

\begin{proposition}
	\label{willemma}
Let $(X,\|\cdot\|)$ be a Banach space, \(\varphi \in C^1(X)\), \(S \subseteq X\), \(c \in \R\), \(\epsilon > 0\) and \(\delta > 0\). Assume that 
for every  \(u \in \varphi^{-1}([c-2\epsilon, c+2\epsilon])\) such that \(B_{2\delta} (u) \cap S \ne \emptyset\) it holds
\[
  \norm{\varphi'(u)}_{X'} \ge \frac{8\epsilon}{\delta}.
\]
Then there exists a homeomorphism \(\eta : X \to X\) such that $\varphi\circ\eta\leq\varphi$ and
\begin{enumerate}[(i)]
 \item \(\eta(u)=u\) if \(\varphi(u) \not \in [c-2\epsilon, c+2\epsilon]\) or \(B_{2\delta}(u) \cap S = \emptyset\);
 \item if \(u \in S\) and \(\varphi(u) \le c+\epsilon\), then \(\varphi(\eta(u)) \le c-\epsilon\);
 \item for every \(u \in X\) it holds \(\norm{\eta(u)-u} \le \delta\).
\end{enumerate}
\end{proposition}

\begin{proof}[Proof of Proposition~\ref{propositionAbstractSymmetricMinimax}]
Let $\gamma\in\Gamma$, $c>a$, $\eps\in ]0, \frac{c-a}{2}[$ and $\delta>0$ be as in the statement of 
Proposition~\ref{propositionAbstractSymmetricMinimax}.
Aiming to apply the quantitative deformation lemma, we set 
\begin{multline*}
  S:=\big\{ w \in \gamma(M):\, \text{for every \(u \in B_{2\delta}(w)\cap \varphi^{-1}([c-2\epsilon, c+2\epsilon])\),}\\
\text{one has \(\norm{\varphi'(u)}_{X'} \ge 8\epsilon/\delta\)} \big\}.
\end{multline*}
In turn, since $S$ fulfills the assumption of Proposition~\ref{willemma}, we get 
a continuous function \(\eta : X \to X\) such that $\varphi\circ\eta\leq\varphi$ which satisfies properties (i)-(iii). Setting 
\[
\Tilde{\gamma}:= \eta \circ \gamma\in\Gamma,
\]
observe that, by virtue of ii), if $t\in M$ and \( \varphi(\Tilde{\gamma}(t)) > c-\epsilon\), then \(\gamma(t) \not \in S\), namely
there exists \(u \in B_{2\delta}(\gamma(t))\) such that $c-2\epsilon\leq \varphi(u)\leq c+2\epsilon$ and 
\(\norm{\varphi'(u)}_{X'} < \frac{8\epsilon}{\delta}\). If we now set 
\[
  \Hat{\gamma}:=\Psi \circ \Tilde{\gamma}\in\Gamma,
\]
we claim that we can find elements \(v \in X\) and \(t \in M\) with the following properties:
\(c-2\epsilon \le \varphi(v) \le c+2\epsilon\),
\(\norm{v-\Hat{\gamma}(t)} \le 2\delta\), \(\varphi(\Hat{\gamma}(t)) > c- \epsilon\) 
and \(\norm{\varphi'(v)}_{X'} < 8 \epsilon /\delta\). 
In fact, if this was not the case,
the assumption of Proposition~\ref{willemma} would be fulfilled with the choice 
$S:=\Hat{\gamma}(M) \cap \varphi^{-1}([c-\frac{\epsilon}{2}, c+\epsilon])$.
We then get a deformation $\hat\eta:X\to X$ such that $\varphi\circ\hat\eta\leq\varphi$ which satisfies properties (i)-(iii). Given now an arbitrary 
element $\tau\in M$, either we have $\varphi(\hat\gamma(\tau))< c-\eps/2$ or 
$$
c-\eps/2\leq \varphi(\hat\gamma(\tau))=\varphi(\Psi(\tilde\gamma(\tau)))\leq \varphi(\tilde\gamma(\tau))\leq \varphi(\gamma(\tau))\leq c+\eps.
$$
In any case, by (ii) and since $\hat\eta\circ\hat\gamma\in\Gamma$
(by i), as $\varphi\circ\hat\gamma|_{M_0}=\varphi\circ\Psi\circ\eta\circ\gamma|_{M_0}\leq\varphi\circ\gamma|_{M_0}\leq a<c-2\eps$), 
$$
c\leq\sup_{M}\varphi(\hat\eta\circ\hat\gamma)\leq c-\eps/2,
$$
yielding a contradiction and proving the claim. Setting \(w := \Tilde{\gamma}(t)\in X\), since 
\(\varphi(\Tilde{\gamma}(t)) \ge \varphi(\Psi(w)) > c-\epsilon\), by the first part of the proof
there exists an element \(u\in X\) with the required properties a.1) and c.1). Furthermore, being
\(u \in B_{2\delta}(\gamma(t))\) and recalling iii), we get
$$
\|u-w\|\leq \|u-\gamma(t)\|+\|\eta(\gamma(t))-\gamma(t)\|\leq 2\delta+\delta,
$$
proving b.1). Analogously, inequalities b.2) and b.3) follow.
\end{proof}

\begin{remark}\rm
	\label{remnonsmooth}
The minimax principle stated in Proposition~\ref{propositionAbstractSymmetricMinimax} for $C^1$ smooth functionals
continues to hold for continuous functionals in the framework of the non-smooth
critical point theory developed in \cites{DM,CDM} by J.N.\ Corvellec, M.\ Degiovanni and M.\ Marzocchi, 
where the quantity $\|\varphi'(u)\|$ is replaced by the notion of weak slope $|d\varphi|(u)\in [0,+\infty]$ (see \cite{DM}*{Definition 2.1}). 
Precisely, the statement of Proposition~\ref{propositionAbstractSymmetricMinimax} in the continuous case
remains the same except the fact that the inequalities $\norm{\varphi'(u)}_{X'} < 8\epsilon/\delta$ and
$\norm{\varphi'(v)}_{X'} < 8\epsilon/\delta$ are replaced by $|d\varphi|(u)< 8\epsilon/\delta$ and 
$|d\varphi|(v)< 8\epsilon/\delta$, respectively. In \cite{corvellec}*{Theorem 2.3}
J.N.\ Corvellec derived a quantitative deformation lemma being the natural non-smooth counterpart of Proposition~\ref{willemma}.
Then, setting
\begin{multline*}
	  A:=\big\{ w \in \gamma(M):\, \text{for every \(u \in B_{2\delta}(w)\cap \varphi^{-1}([c-2\epsilon, c+2\epsilon])\)}\\
	\text{one has $|d\varphi|(u) \ge 8\epsilon/\delta$} \big\}.
\end{multline*}
	by applying \cite{corvellec}*{Theorem 2.3} to the set $A$ (or slightly modifying the argument if $A$ is not closed in $X$)
	the same conclusion in the first part of the proof of Proposition~\ref{propositionAbstractSymmetricMinimax} is obtained.
	In a similar fashion, also the second part of the proof can be proved reusing~\cite{corvellec}*{Theorem 2.3}.
	For applications of non-smooth critical point theory to various classes of quasi-linear elliptic PDEs, we refer the interested reader 
	to the monograph \cite{squasbook}. In the recent work \cite{Squassina2011-1} a symmetric minimax theorem is obtained for a class
	of lower semi-continuous functionals of the form $\varphi(u)=\int_\Omega j(u,|Du|)-\int_\Omega G(|x|,u)$.
\end{remark}

\section{Application to the Mountain Pass lemma}
\label{sectionMountain}
We will now apply the result of the previous section in order to 
prove the result announced in the introduction.

\begin{proposition}
	\label{mpl}
Assume that \(\varphi \in C^1(H^1_0(\Omega))\), \(0\) is a strict local minimum of \(\varphi\) which is not a global minimum  and define
\[
  \Gamma=\{ \gamma \in C([0, 1], H^1_0(\Omega)):\, \gamma(0) = 0 \text{ and } \varphi(\gamma(1)) < 0\}
\]
and
\[
  c=\inf_{\gamma \in \Gamma} \sup_{t\in [0,1]} \varphi (\gamma(t)).
\]
Assume that \(\varphi\) satisfies the Palais-Smale condition.
Let \(H\) be a closed half-space with \(\sigma_H (\Omega) = \Omega\) and for 
every \(u \in H^1_0(\Omega)\), \(\varphi (u^H) \le \varphi (u) \).
If \((\gamma_n)_{n \in \N}\) is a sequence in \(\Gamma\) such that
\[
  \limsup_{n \to \infty} \sup_{t\in [0,1]} \varphi(\gamma_n([0, 1])) \le c.
\]
then there exists \(u \in H^1_0(\Omega)\) such that \(\varphi (u) = \varphi (u^H) = c\), \(\varphi' (u) = \varphi' (u^H) = 0\) and
\[
 \liminf_{n \to \infty} \dist_{H^1} (u, \gamma_n([0, 1]) ) = 0 .
\]
\end{proposition}
\begin{proof}
Notice that the map \(\Psi : H^1_0(\Omega) \to H^1_0(\Omega)\) defined by \(\Psi (u) := u^H\) is 
continuous by \citelist{\cite{VanSchaftingen2005}*{Proposition 2.5}\cite{VanSchaftingen2006}*{Corollary 2.40}}. 
By assumption, we have $\varphi\circ \Psi\leq \varphi$ and $\Psi(\gamma)\in\Gamma$, for all $\gamma\in\Gamma$, where 
$\Psi(\gamma)(t):=\Psi(\gamma(t))$ for $t\in [0,1]$. Without loss of generality, we can assume that 
\[
 \sup_{t\in [0,1]} \varphi(\gamma_n([0, 1])) \le c + \frac{1}{n^2}.
\]
Apply now Proposition~\ref{propositionAbstractSymmetricMinimax} with the choice \(M := [0, 1]\), \(M_0 := \{0, 1\}\),
\(\delta =\delta_n:= \frac{1}{n}\), \(\eps =\eps_n:= \frac{1}{n^2}\) and
\[
  \Gamma_0 := \{ \gamma_0 \in C(\{0, 1\}, H^1_0(\Omega)):\, \gamma_0(0) = 0 \text{ and } \varphi(\gamma_0(1)) < 0\}.
\]
One then obtains three sequences \( (u_n)_{n \in \N}\), \( (v_n)_{n \in \N}\) 
and \( (w_n)_{n \in \N}\) in \(H^1_0(\Omega)\) such that
\begin{gather*}
 \lim_{n\to \infty} \varphi (u_n) = \lim_{n\to \infty} \varphi (v_n) = c,\\
 \lim_{n\to \infty} \varphi' (u_n) = \lim_{n\to \infty} \varphi' (v_n) = 0,\\
 \lim_{n\to \infty} \norm{u_n - w_n}_{H^1} = \lim_{n\to \infty} \dist_{H^1} (w_n, \gamma_n([0, 1])) =\lim_{n\to \infty} \norm{v_n - w_n^H}_{H^1} = 0.
\end{gather*}
Since \(\varphi\) satisfies the Palais-Smale condition, up to a subsequence, \( (u_n)_{n \in \N}\) converges
to some \(u \in H^1_0(\Omega)\). Hence, the sequence \((w_n)_{n \in \N}\) also converges to \(u\). By continuity 
of the polarization, \((v_n)_{n \in \N}\)  converges to \(u^H\). The rest follow by the fact that $\varphi$ is $C^1(H^1_0(\Omega))$.
\end{proof}

\begin{remark}\rm
As pointed out in Remark~\ref{remnonsmooth}, the shadowing minimax principle 
in Proposition~\ref{propositionAbstractSymmetricMinimax} extends to the case
of continuous functionals in the framework of the non-smooth critical point theory of \cites{DM,CDM}
replacing $\|\varphi'(u)\|$ by the weak slope $|d\varphi|(u)$ \cite{DM}*{Definition 2.1}. In this setting,
the Palais-Smale condition has to be read as follows: if \((u_n)_{n \in \N}\) is a sequence in \(H^1_0(\Omega)\) such 
that \((\varphi(u_n))_{n \in \N}\) converges and \(|d\varphi|(u_n) \to 0\) as \(n \to \infty\), then 
\((u_n)_{n \in \N}\) converges strongly, up to a subsequence, to some $u$ in $H^1_0(\Omega)$.
Therefore, taking into account that $\varphi$ is continuous and
the map $H^1_0(\Omega)\ni u\mapsto |d\varphi|(u)\in [0,+\infty]$ is in turn lower
semi-continuous \cite{DM}*{Proposition 2.6}, Proposition~\ref{mpl} holds true for continuous functionals, with essentially the same proof, 
by replacing the conclusion that $\varphi'(u)=0$ and $\varphi'(u^H)=0$ with $|d\varphi|(u)=0$ and $|d\varphi|(u^H)=0$ respectively. 
For many continuous functionals of the Calculus of Variations this implies \cite{squasbook} that $u$ and $u^H$ are distributional solutions
of the associated Euler-Lagrange equation. 
\end{remark}

\begin{remark}\rm
Up to slight modifications, Proposition~\ref{mpl} holds also when the assumption
that the closed half-space \(H\) is axially symmetric, that is \(\sigma_H (\Omega) = \Omega\), is replaced by the more
general assumption that $0\in H$ and $\sigma^H(\Omega)=\Omega$, denoting $\sigma^H(\Omega)$ the polarized
domain of $\Omega$, namely the unique domain satisfying $\chi_{\sigma^H(\Omega)}=(\chi_\Omega)^H$.
If for instance $0\not\in\partial H$, then $\sigma_H (B_1(0)) \neq B_1(0)$ but instead $\sigma^H(B_1(0))=B_1(0)$.
\end{remark}

\vskip15pt
\noindent
{\bf Acknowledgments.} The first author was supported by the 2009 Italian PRIN project: {\em Metodi Variazionali e Topologici
nello Studio di Fenomeni non Lineari}. The work was partially done while the second author was visiting the University of Verona
during winter 2011.

\bigskip


\medskip
\begin{bibdiv}
\begin{biblist}

\bib{AmbrosettiRabinowitz}{article}{
   author={Ambrosetti, Antonio},
   author={Rabinowitz, Paul H.},
   title={Dual variational methods in critical point theory and
   applications},
   journal={J. Functional Analysis},
   volume={14},
   date={1973},
   pages={349--381},
}


\bib{BartschWethWillem}{article}{
   author={Bartsch, Thomas},
   author={Weth, Tobias},
   author={Willem, Michel},
   title={Partial symmetry of least energy nodal solutions to some
   variational problems},
   journal={J. Anal. Math.},
   volume={96},
   date={2005},
   pages={1--18},
   issn={0021-7670},
}

\bib{BrockSolynin2000}{article}{
   author={Brock, Friedemann},
   author={Solynin, Alexander Yu.},
   title={An approach to symmetrization via polarization},
   journal={Trans. Amer. Math. Soc.},
   volume={352},
   date={2000},
   number={4},
   pages={1759--1796},
   issn={0002-9947},
}


\bib{BrothersZiemer1988}{article}{
   author={Brothers, John E.},
   author={Ziemer, William P.},
   title={Minimal rearrangements of Sobolev functions},
   journal={J. Reine Angew. Math.},
   volume={384},
   date={1988},
   pages={153--179},
   issn={0075-4102},
}


\bib{BJMar}{article}{
   author={Byeon, J.},
   author={Jeanjean, L.},
   author={Maris, M.},
   title={Symmetry and monotonicity of least energy solutions},
   journal={Calc. Var. Partial Differential Equations},
   volume={36},
   date={2009},
   number={4},
   pages={481--492},
}


\bib{corvellec}{article}{
   author={Corvellec, J.N.},
   title={Quantitative deformation theorems and critical point theory},
   journal={Pacific J. Math.},
   volume={187},
   date={1999},
   number={2},
   pages={263--279},
}


\bib{CDM}{article}{
   author={Corvellec, J.N.},
author={Degiovanni, M.},
author={Marzocchi, M.},
   title={Deformation properties for continuous functionals and critical point theory},
   journal={Topol. Methods Nonlinear Anal.},
   volume={1},
   date={1993},
   number={1},
   pages={151--171},
}


\bib{DM}{article}{
author={Degiovanni, M.},
author={Marzocchi, M.},
   title={A critical point theory for nonsmooth functionals},
   journal={Ann. Mat. Pura Appl.},
   volume={167},
   date={1994},
   number={4},
   pages={73--100},
}


\bib{GidasNiNirenberg1979}{article}{
   author={Gidas, B.},
   author={Ni, Wei Ming},
   author={Nirenberg, L.},
   title={Symmetry and related properties via the maximum principle},
   journal={Comm. Math. Phys.},
   volume={68},
   date={1979},
   number={3},
   pages={209--243},
   issn={0010-3616},
}

\bib{GladialiPacellaWeth2010}{article}{
   author={Gladiali, Francesca},
   author={Pacella, Filomena},
   author={Weth, Tobias},
   title={Symmetry and nonexistence of low Morse index solutions in
   unbounded domains},
   journal={J. Math. Pures Appl. (9)},
   volume={93},
   date={2010},
   number={5},
   pages={536--558},
   issn={0021-7824},
}



\bib{Kawohl1985}{book}{
   author={Kawohl, Bernhard},
   title={Rearrangements and convexity of level sets in PDE},
   series={Lecture Notes in Mathematics},
   volume={1150},
   publisher={Springer-Verlag},
   place={Berlin},
   date={1985},
   pages={iv+136},
   isbn={3-540-15693-3},
}

\bib{Lopes1996}{article}{
   author={Lopes, Orlando},
   title={Radial symmetry of minimizers for some translation and rotation
   invariant functionals},
   journal={J. Differential Equations},
   volume={124},
   date={1996},
   number={2},
   pages={378--388},
   issn={0022-0396},
}

\bib{PacellaWeth}{article}{
   author={Pacella, Filomena},
   author={Weth, Tobias},
   title={Symmetry of solutions to semilinear elliptic equations via Morse
   index},
   journal={Proc. Amer. Math. Soc.},
   volume={135},
   date={2007},
   number={6},
   pages={1753--1762},
   issn={0002-9939},
}

\bib{PolyaSzego}{book}{
   author={P{\'o}lya, G.},
   author={Szeg{\H o}, G.},
   title={Isoperimetric Inequalities in Mathematical Physics},
   series={Annals of Mathematics Studies},
   volume = {27},
   publisher={Princeton University Press},
   place={Princeton, N.J.},
   date={1951},
   pages={xvi+279},
}

\bib{Squassina2011-1}{article}{
  author = {Squassina, M.},
  title = {Radial symmetry of minimax critical points for non-smooth functionals},
  journal = {Commun. Contemp. Math.},
  volume = {13},
  date = {2011},
 number={3},
pages={487--508},
}

\bib{Squassina2011}{article}{
  author = {Squassina, M.},
  title = {On Ekeland's variational principle},
  journal = {J. Fixed Point Theory Appl.},
  volume = {10},
  date = {2011},
number={1},
pages={191--195},
}

\bib{Squassina2012}{article}{
  author = {Squassina, M.},
  title = {Symmetry in variational principles and applications}, 
  journal = {J. London Math. Soc.},
  volume = {85},
  year = {2012},
  pages = {to appear},
}


\bib{squasbook}{book}{
    AUTHOR = {Squassina, M.},
     TITLE = {Existence, multiplicity, perturbation, and concentration
              results for a class of quasi-linear elliptic problems},
   series = {Electronic Journal of Differential Equations Monographs},
    volume = {7},
 PUBLISHER = {Texas State University--San Marcos, Department of Mathematics},
      YEAR = {2006},
     PAGES = {front matter+213 pp. (electronic)},
}


\bib{VanSchaftingen2005}{article}{
   author={Van Schaftingen, Jean},
   title={Symmetrization and minimax principles},
   journal={Commun. Contemp. Math.},
   volume={7},
   date={2005},
   number={4},
   pages={463--481},
   issn={0219-1997},
}

\bib{VanSchaftingen2006}{article}{
   author={Van Schaftingen, Jean},
   title={Approximation of symmetrizations and symmetry of critical points},
   journal={Topol. Methods Nonlinear Anal.},
   volume={28},
   date={2006},
   number={1},
   pages={61--85},
   issn={1230-3429},
}

\bib{VanSchaftingen2009}{article}{
   author={Van Schaftingen, Jean},
   title={Explicit approximation of the symmetric rearrangement by
   polarizations},
   journal={Arch. Math. (Basel)},
   volume={93},
   date={2009},
   number={2},
   pages={181--190},
   issn={0003-889X},
}


\bib{VanSchaftingenWillem}{article}{
   author={Van Schaftingen, Jean},
   author={Willem, Michel},
   title={Symmetry of solutions of semilinear elliptic problems},
   journal={J. Eur. Math. Soc. (JEMS)},
   volume={10},
   date={2008},
   number={2},
   pages={439--456},
   issn={1435-9855},
}


\bib{WillemMinimax}{book}{
   author={Willem, Michel},
   title={Minimax theorems},
   series={Progress in nonlinear differential equations and their applications}, 
   volume={24},
   publisher={Birkh\"auser}, 
   address={Boston -- Basel -- Berlin},
   date={1996},
}

\end{biblist}
\end{bibdiv}
\end{document}